\documentclass[12pt,leqno,a4paper]{amsart}
\usepackage{amsmath,amsthm,amssymb}

\usepackage{enumerate,hyperref}
\usepackage{times}

\newcommand{\CC}{{\mathbb{C}}}
\newcommand{\FF}{{\mathbb{F}}}
\newcommand{\sS}{{\mathbb{S}}}

\newcommand{\bG}{{\mathbf{G}}}
\newcommand{\bZ}{{\mathbf Z}}

\newcommand{\cE}{{\mathcal{E}}}
\newcommand{\cP}{{\mathcal P}}

\newcommand{\Cl}{\mathfrak{Cl}}

\newcommand{\Bl}{\operatorname{Bl}}
\newcommand{\bl}{\operatorname{bl}}
\newcommand{\Syl}{\operatorname{Syl}}

\newcommand{\IBr}{\operatorname{IBr}}
\newcommand{\Irr}{\operatorname{Irr}}
\newcommand{\Cent}{\ensuremath{{\mathrm{C}}}}
\newcommand{\GL}{{\operatorname{GL}}}
\newcommand{\SL}{{\operatorname{SL}}}
\newcommand{\Sp}{{\operatorname{Sp}}}
\newcommand{\GO}{{\operatorname{GO}}}
\newcommand{\GU}{{\operatorname{GU}}}
\newcommand{\SU}{{\operatorname{SU}}}
\newcommand{\Spin}{{\operatorname{Spin}}}
\newcommand{\dz}{{\operatorname{dz}}}
\newcommand{\rdz}{{\operatorname{rdz}}}
\newcommand{\tw}[1]{{}^#1\!}

\newcommand{\w}{\widetilde}
\newcommand{\wh}{\widehat}
\newcommand{\whS}{{\widehat S}}
\newcommand{\whX}{{\widehat X}}
\let\sbs=\subseteq
\def\nor{{\triangleleft\,}}
\def\zent#1{{\bf Z}(#1)}
\def\cent#1#2{{\bf C}_{#1}(#2)}
\def\norm#1#2{{\bf N}_{#1}(#2)}

\renewcommand{\o}{\overline}

\makeatletter
\def\Lset#1{\Lset@h#1@}
\def\Lset@h#1@{\left\{#1\right\}}
\def\spann<#1>{\left\langle#1\right\rangle}

\newtheorem{thm}[subsection]{Theorem}
\newtheorem{prop}[subsection]{Proposition}
\newtheorem{cor}[subsection]{Corollary}
\newtheorem{lem}[subsection]{Lemma}

\numberwithin{equation}{section}

\begin{document}

\title[Invariant blocks under coprime actions]
  {Invariant blocks under coprime actions}

\date{\today}

\author{Gunter Malle}
\address{FB Mathematik, TU Kaiserslautern, Postfach 3049,
 67653 Kaisers\-lautern, Germany.}
\email{malle@mathematik.uni-kl.de}

\author{Gabriel Navarro}
\address{Departament d'\`Algebra, Universitat de Val\`encia,
 Dr. Moliner 50, 46100 Burjassot, Spain.}
\email{gabriel.navarro@uv.es}

\author{Britta Sp\"ath}
\address{FB Mathematik, TU Kaiserslautern, Postfach 3049,
 67653 Kaisers\-lautern, Germany.}
\email{spaeth@mathematik.uni-kl.de}

\thanks{The research of the first and third author has been supported
  by the ERC Advanced Grant 291512}

\keywords{Coprime action, blocks, characters}

\subjclass[2010]{20C15, 20C20, 20C33}

\begin{abstract}
If a finite group $A$ acts coprimely as automorphisms on a finite group $G$,
then the $A$-invariant Brauer $p$-blocks of $G$ are exactly those that
contain  $A$-invariant  irreducible characters.
\end{abstract}

\maketitle

\section{Introduction}

One of the most basic situations in group theory is when a group $A$ acts by
automorphisms on another group $G$. If we further assume that $A$ and $G$ are
finite   of coprime orders, it is well-known that most of the representation
theory of $G$ admits a version in which only the $A$-invariant structure is
taken into account. For instance, it is true (and not trivial) that the
number of irreducible complex characters of $G$ which are fixed by $A$ equals
the number of conjugacy classes of $G$ fixed by~$A$.

Although it is fair to say that the ordinary $A$-representation theory of $G$
is mostly well developed, we cannot say the same about $A$-invariant modular
representation theory (that is, of prime characteristic $p$). For instance,
it is suspected that the number of $A$-invariant irreducible $p$-Brauer
characters of $G$ is the number of $A$-invariant $p$-regular classes of $G$,
but this conjecture continues to be open. This, together with some other
problems, was proposed more than 20 years ago in \cite{Nav94}.

The extensive research during these years on Brauer $p$-blocks allows us now
to give a solution to Problem 6 of \cite{Nav94}.

\begin{thm}   \label{main}
 Suppose that the finite group $A$ acts by automorphisms on the finite group
 $G$ with $(|A|,|G|)=1$. Let $p$ be a prime, and let $B$ be a Brauer $p$-block
 of $G$. Then the following statements are equivalent:
 \begin{itemize}
  \item[\rm(i)] $B$ is $A$-invariant,
  \item[\rm(ii)] $B$ contains some $A$-invariant character $\chi\in\Irr(G)$,
  \item[\rm(iii)] $B$ contains some $A$-invariant Brauer character
    $\phi\in\IBr(G)$.
 \end{itemize}
\end{thm}

Of course, (ii) and (iii) easily imply that $B$ is $A$-invariant, so all the
work is concentrated in proving that (i) implies (ii) and (iii).

The paper is structured in the following way.
In Section \ref{sec_quasi_simple} a version of the main theorem is proven in
the case where $G$ is quasi-simple. Afterwards in
Section~\ref{sec_Gal_type_thm} we present a Gallagher type theorem for blocks.
In connection with Dade's ramification group from \cite{Dade_blockextensions},
revisited in Section~\ref{sec_Dade_group}, we show the existence of character
triple isomorphisms having crucial properties with respect to coprime
action and blocks (see Section~\ref{sec_char_trip}). We conclude in the final
section with the reductions proving how the results on quasi-simple groups
imply our main statement.

\section{Quasi-simple groups and their central products}\label{sec_quasi_simple}

The aim of this section is the proof of a strengthened version of
Theorem~\ref{main} in cases where the group $G/\bZ(G)$ is the direct product
of $r$ isomorphic non-abelian simple groups. First we deal with the case
where $G$ is the universal covering group of a simple group and hence $G$
is a quasi-simple group.

In the following we use the standard notation around characters and blocks
as introduced in \cite{Isa} and \cite{Navarro}. Let $p$ be a prime.
If $A$ acts on $G$, we denote by $\Bl_A(G)$ the set of $A$-invariant
$p$-blocks of $G$. If $Z\nor G$, $B\in\Bl(G)$ and $\nu\in\Irr(Z)$ we
denote by $\Irr(B| \nu)$ the set $\Irr(B)\cap \Irr(G| \nu)$. Also, sometimes
we will work in $GA$, the semidirect product of $G$ with $A$.

\begin{thm}  \label{thmA_simple}
 Let $G$ be the universal covering group of a non-abelian simple group $S$,
 $A$ a group acting on $G$ with $(|G|,|A|)=1$, $B$ an $A$-invariant $p$-block
 of $G$, $Z$ the Sylow $p$-subgroup of $\bZ(G)$ and $\nu\in\Irr(Z)$
 $A$-invariant.
 Then $B$ contains some $A$-invariant character $\chi\in\Irr(G| \nu)$.
\end{thm}

This theorem is true whenever $B$ is the principal block and $Z=1$, since
then $\chi$ can be chosen to be the trivial character.
On the other hand when $B$ is a block of central defect, $\Irr(B|\nu)$
contains exactly one character and hence this one is $A$-invariant.

Note that neither alternating nor sporadic simple groups possess coprime
automorphisms, and that for groups of Lie type the only coprime order
automorphisms are field automorphisms (up to conjugation). Thus, for the proof
of the theorem, we can assume
that $S$ is simple of Lie type. We consider the following setup. Let $\bG$
be a simple algebraic group of simply connected type over an algebraic
closure of a finite field of characteristic~$r$, and let $F:\bG\rightarrow\bG$
be a Steinberg endomorphism, with group of fixed points $G:=\bG^F$. It is
well known that all finite simple groups of Lie type occur as $G/\bZ(G)$
with $G$ as before, except for the Tits group $\tw2F_4(2)'$. Since the latter
does not possess coprime automorphisms, we need not consider it here.
Furthermore, in all but finitely many cases, the group $G$ is the universal
covering group of $S=G/\bZ(G)$. None of the exceptions, listed for example in
\cite[Table~24.3]{MT}, has coprime automorphisms, apart from the Suzuki
group $\tw2B_2(8)$. But for $2^2.\tw2B_2(8)$, the outer automorphism of order
three permutes the three non-trivial central characters, hence our claim
holds.
For the proof of Theorem~\ref{thmA_simple} we may thus assume that $G$ and
$S$ are as above, and that $A$ induces a (necessarily cyclic) group of coprime
field automorphisms on $G$.

We first discuss the action of such automorphisms on the Lusztig series
$\cE(G,s)\subseteq\Irr(G)$ of irreducible characters of $G$, where $s$ runs
over semisimple elements of a dual group $G^*={\bG^*}^F$ of $G$. Let
$\bG\hookrightarrow\tilde\bG$ be a regular embedding, that is, $\tilde\bG$ is
connected reductive with connected center and with derived subgroup $\bG$.
Corresponding to this there exists a surjection $\tilde\bG^*\rightarrow\bG^*$
of dual groups. Note that all field automorphisms of $\bG$ are induced by
those of $\tilde\bG$. Let $\gamma$ be a field automorphisms of $\tilde\bG$.
We denote the corresponding field automorphism of $\tilde\bG^*$ also by
$\gamma$. Let $\tilde s\in\tilde G^*:=\tilde\bG^{*F}$ be semisimple. Now by
\cite[Prop.~3.5]{MN11}, $\gamma$ acts trivially on $\cE(\tilde G,\tilde s)$
whenever it stabilizes $\cE(\tilde G,\tilde s)$. Let $s\in G^*$ with preimage
$\tilde s$. Then by definition $\cE(G,s)$ consists of the constituents of
the restrictions of characters in $\cE(\tilde G,\tilde s)$ to $G$.

\begin{lem}   \label{lem:fld}
 In the above setting let $\gamma$ be a coprime (field) automorphism of $G$.
 If $\gamma$ stabilizes $\cE(G,s)$, then it fixes $\cE(G,s)$ pointwise.
\end{lem}

\begin{proof}
According to what we said before, $\gamma$ can only permute the
$G$-constituents of a fixed character $\chi\in \cE(\tilde G,\tilde s)$.
But the number of such constituents is bounded by $|\bZ(G)|$, and it is easily
checked that all primes not larger than $|\bZ(G)|$ divide $|G|$, so that all
prime divisors of the order of $\gamma$ are larger than the number of such
constituents. Thus the action has to be trivial.
\end{proof}

There are two quite different types of behaviour now. Either $p$ is the defining
characteristic of $G$, then coprime field automorphisms fix all $p$-blocks
(but certainly not all irreducible characters); or $p$ is different from
the defining characteristic, in which case all characters in an invariant
block are fixed individually (but not all $p$-blocks are invariant):

\begin{prop}
 In the above situation, Theorem~\ref{thmA_simple} holds when $p$ is the
 defining characteristic of $G$.
\end{prop}

\begin{proof}
In this case the $p$-blocks of positive defect of a group of Lie type $G$ are
in bijection with the characters of $\bZ(G)$, by a result of Humphreys
\cite{Hum}. Since the claim is certainly true for the principal block, we may
assume that $\bZ(G)\ne 1$, and so in particular $G$ is not a Suzuki or Ree
group. For each type of group and each $\gamma$-stable $1\ne\nu\in\bZ(G)$ we
give in Table~\ref{tab:defchar}
a semisimple element $s$ of the dual group $G^*$ of $G$ with the following
properties: the Lusztig series $\cE(G,s)$ of irreducible characters lies in
$\Irr(G|\nu)$, and the class of $s$ is $\gamma$-invariant (since $s$
corresponds to the $\gamma$-stable central character $\nu$).
It then follows that $\cE(G,s)$ is stable under all field automorphisms of
$G$, and this implies by Lemma~\ref{lem:fld} that the characters in $\cE(G,s)$
are individually stable, hence provide characters as claimed.
\end{proof}

\begin{table}[htbp]
\caption{Semisimple elements}   \label{tab:defchar}
\[\begin{array}{|r|r|l|l|}
\hline
 G& C_{G^*}(s)& o(\nu)& \text{conditions}\\
\hline
   \SL_n(q)& \GL_{n-1}(q)& \text{divides }(n,q-1)& (n,q-1)>1\\
   \SU_n(q)& \GU_{n-1}(q)& \text{divides }(n,q+1)& (n,q+1)>1\\
 \Spin_{2n+1}(q)& C_{n-1}& 2& q\text{ odd}\\
     \Sp_{2n}(q)& \GO_{2n}^\pm(q)& 2& q\text{ odd}\\
 \Spin_{2n}^\pm(q)& B_{n-1}& \text{divides }4& q\text{ odd}\\
     E_6(q)&     D_5& 3& q\equiv1\pmod3\\
 \tw2E_6(q)& \tw2D_5& 3& q\equiv2\pmod3\\
     E_7(q)&     E_6& 2& q\equiv1\pmod4\\
     E_7(q)& \tw2E_6& 2& q\equiv3\pmod4\\
\hline
\end{array}\]
\end{table}

Let us now turn to the case where $p$ is not the defining characteristic
of $G$.

\begin{prop}   \label{prop:crosschar}
 In the above situation, when $p$ is different from the defining
 characteristic of $G$, if $B$ is a $\gamma$-invariant $p$-block of $G$, then
 all $\chi\in\Irr(B)$ are fixed by $\gamma$. In particular
 Theorem~\ref{thmA_simple} holds in this case.
\end{prop}

\begin{proof}
Let $B$ be a $p$-block of $G$, $Z=\bZ(G)_p$ and $\nu\in\Irr(Z)$. Let $\gamma$
be a coprime (field) automorphism of $G$ fixing $B$ and $\nu$. By a result
of Brou\'e and Michel there exists a semisimple $p'$-element $s\in G^*$ such
that $B\subseteq \cE_p(G,s)$. Let $G\hookrightarrow\tilde G$ be a regular
embedding, with corresponding epimorphism $\tilde G^*\rightarrow G^*$ of dual
groups. Let $\tilde s\in \tilde G^*$ be a preimage of $s$. Since the class of
$s$ is $\gamma$-stable, the same argument as in the proof of Lemma~\ref{lem:fld}
shows that the class of $\tilde s$ is also $\gamma$-stable. Since the
centralizer of $\tilde s$ is connected, this means that we may assume without
loss of generality that $\tilde s$ itself is $\gamma$-stable, and so is
$C_{\tilde G^*}(\tilde s)$. Now consider $H:=(\tilde G^*)^\gamma$, the
fixed point subgroup of $\tilde G^*$ under $\gamma$. This is again a group
of Lie type, of the same type as $\tilde G^*$. Since $p$ divides $|\bZ(G)|$ by
assumption,
it also divides the order of the Weyl group of $G$. By \cite[Prop.~3.12]{MN11},
then $C_H(\tilde s)=C_{\tilde G^*}(\tilde s)^\gamma$ contains a Sylow
$p$-subgroup of $C_{\tilde G^*}(\tilde s)$. In particular, every semisimple
$p$-element in $C_{\tilde G^*}(\tilde s)$ has a $\gamma$-stable conjugate
$\tilde t\in\tilde G^*$. Then $\cE(G,\tilde s\tilde t)$ is $\gamma$-stable,
which implies that $\cE(G,st)$ is $\gamma$-stable, and hence fixed pointwise
by $\gamma$, by Lemma~\ref{lem:fld}. Thus, all elements of $B\cap\cE_p(G,s)$
are fixed by $\gamma$, as claimed.
\end{proof}

We next prove an analogous result for Brauer characters:

\begin{thm}  \label{thm_IBr_simple}
 Let $G$ be the universal covering group of a non-abelian simple group $S$,
 $A$ a group acting on $G$ with $(|G|,|A|)=1$, and $B\in\Bl_A(G)$. Then
 there exists an $A$-invariant Brauer character $\phi\in\IBr(B)$.
\end{thm}

\begin{proof}
As argued in the proof of Theorem~\ref{thmA_simple} we may assume that
$S$ is of Lie type and $A$ induces coprime field automorphisms. Moreover,
$G$ is not an exceptional covering group of $S$. First assume
that $p$ is the defining characteristic of $G$. Let $B$ be an $A$-invariant
block of $G$, corresponding to the central character $\nu$ of $G$. Then there
is a faithful irreducible Brauer character $\phi$ of $G/\ker(\nu)$
corresponding to a suitable fundamental weight $\omega$ of the underlying
algebraic group as given in Table~\ref{tab:defIBr}. If $G$ is untwisted,
defined over $\FF_q$ with $q=p^f$, then a generator $\gamma$ of $A$ has
order $a$ with $f=ka$. By Steinberg's tensor product theorem
(see \cite[Thm.~16.12]{MT}) then $\phi':=\bigotimes_{i=0}^{a-1}\gamma^i(\phi)$
is an irreducible Brauer character of $G$ corresponding to the weight
$\sum_{i=0}^{a-1}p^{ik}\omega$, which is $\gamma$-invariant. It lies over
the character $\nu^d$ of $\bZ(G)$, with $d=\sum_{i=0}^{a-1}p^{ik}$. Since
$|\bZ(G)|$ divides $p^k-1$ and $a$ is prime to $|\bZ(G)|$, this is again a
faithful character of $\bZ(G)$. Thus, this construction yields an invariant
Brauer character in the $p$-block lying above $\nu^d$. Starting instead with
$\phi$ in the $p$-block above $\nu^c$, with $cd\equiv1\pmod{o(\nu)}$, we find
an invariant character in $B$.

\begin{table}[htbp]
\caption{Faithful Brauer characters}   \label{tab:defIBr}
\[\begin{array}{|r|r|l|l|}
\hline
 G& \phi(1)& \text{weight}& \text{conditions}\\
\hline
   \SL_n(q),\SU_n(q)& \binom{n}{i}& \omega_i& (n,q\pm1)>1\\
 \Spin_{2n+1}(q)& 2^n& \omega_n& q\text{ odd}\\
     \Sp_{2n}(q)& 2n& \omega_1& q\text{ odd}\\
 \Spin_{2n}^\pm(q)& 2n,2^{n-1}& \omega_1,\omega_{n-1},\omega_n& q\text{ odd}\\
 E_6(q),\tw2E_6(q)&     27& \omega_1& 3{\not|}q\\
     E_7(q)&     56& \omega_7& q\text{ odd}\\
\hline
\end{array}\]
\end{table}

If $G$ is twisted, we may assume that it is not very twisted and that the
twisting has order~2 (since else there is just one $p$-block of positive
defect). So $G$ is defined over $\FF_{q^2}$, with $q=p^f$ and $a|f$. The same
argument as before applies in this case as well.
\par
Now assume that $p$ is different from the defining characteristic of $G$,
and let $B$ be an $A$-invariant $p$-block. Then $B$ is contained in
$\cE_p(G,s)$ for some semisimple element $s\in G^*$, and we showed in
Proposition~\ref{prop:crosschar} that all elements in $B$ are fixed by $A$.
Since any irreducible Brauer character in $B$ is an integral linear
combination of ordinary irreducible characters in $B$ restricted to
$p'$-classes, it follows that $\IBr(B)$ is fixed point-wise by $A$ as well.
\end{proof}

We conclude this section with the analogous result on central products of
quasi-simple groups.

\begin{cor}   \label{cor_dir_prod}
 Let $G$ be a finite group and let $A$ act on $G$ with $(|G|,|A|)=1$.
 Assume that $G/\bZ(G)$ is the direct product of $r$ isomorphic non-abelian
 simple groups that are transitively permuted by $A$. Let $B\in \Bl(G)$ be
 $A$-invariant.
 \begin{enumerate}[\rm (a)]
  \item Let $Z\in\Syl_p(\bZ(G))$ and $\nu\in\Irr(Z)$. Then $B$ contains
   some $A$-invariant $\chi\in\Irr(G| \nu)$.
  \item $B$ contains some $A$-invariant $\phi\in\IBr(G)$.
 \end{enumerate}
\end{cor}

\begin{proof}
Both parts can be shown analogously. We give here the proof of part (a).
Let $S$ be the simple non-abelian group such that $G/\bZ(G)$ is isomorphic to
the $r$-fold direct product of groups isomorphic to $S$.

It suffices to prove the statement in the case where $G$ is perfect.
Indeed, we have $G=[G,G]\bZ(G)$ by the given structure of $G$.
Assume that the statement holds in the case of perfect groups.
Hence the block $B'\in\Bl([G,G])$ covered by $B$ has then an $A$-invariant
character $\chi_0$ lying over $\nu_{\bZ(G)\cap [G,G]}$. The character
$\chi_0\cdot \nu$ defined as in \cite[Section 5]{IMN} as the unique character
in $\Irr(G| \chi_0)\cap \Irr(G| \nu)$ is then necessarily $A$-invariant.

In the following we consider the case where $G$ is perfect. Accordingly $G$
has a universal covering group, namely
\[ X:=\whS\times \cdots  \times \whS=\whS^r \text{ ($r$ factors) },\]
where $\whS$ is the universal covering group of $S$. Let
$\epsilon: X\rightarrow G$ be the associated canonical epimorphism.
Because of \cite[5.1.4]{GLS3} there is a canonical action of $A$ on $X$ induced
by the action of $A$ on $G$, such that $\epsilon$ is $A$-equivariant.
Note that the action of $A$ on $X$ is by definition coprime.

For $1\leq i \leq r$ let $X_i:=\{(1,\ldots,1,x,1,\ldots,1)\mid x\in \whS\}$,
which is canonically isomorphic to $\whS$. By assumption $A$ acts transitively
on the set of groups $X_i$. So for every $1\leq i \leq r$ there exists an
element $a_i\in A$ such that $X_1^{a_i}=X_i$.

The character $\nu\in\Irr(Z)$ can be uniquely extended to a character
$\w\nu\in\Irr(\bZ(G))$ such that $\Irr(B | \w\nu)\neq\emptyset$. Via
$\epsilon$ the character $\w\nu$ corresponds to some character
$\wh \nu\in\Irr(\bZ(X))$ that can be written as
\[ \wh \nu=\wh\nu_1\times\cdots \times \wh \nu_r \]
for suitable $\wh \nu_i\in\Irr(\bZ(\whS))$

The block $B$ corresponds to a unique block $\wh B\in\Bl(X)$, see
\cite[(9.9) and (9.10)]{Navarro}. Accordingly $\wh B$ is $A$-invariant and
can be written
as $\wh B= \wh B_1\times \cdots \times \wh B_r$ where $\wh B_i\in\Bl(\wh S)$.
The action of $A_i$ on $X$ induces then a coprime action of $A_i$ on $\wh S$
stabilizing $\wh B_i$. According to Theorem \ref{thmA_simple} there exists an
$A_1$-invariant character $\psi_1\in\Irr(\wh B_1| \wh \nu_1)$.

Via the canonical isomorphism between $\whX_i$ and $\whS$ the character
$\psi_i:=(\psi_1)^{a_i}\in\Irr(\whS)$ is well-defined. Since $\wh B$ and
$\wh \nu$ are $A$-invariant the character $\psi_i$ belongs to
$\Irr(\wh B_i| \wh \nu_i)$. Accordingly the character
$\psi:=\psi_1\times \cdots \times \psi_r$ belongs to $\Irr(\wh B|\wh \nu)$.

In the next step we prove that $\psi$ is $A$-invariant. Let $\phi_i$ be the
character of $X_i$ corresponding to $\psi_i$ via the canonical isomorphism.
Then it is sufficient to prove that for any $a\in A$ and $1\leq i,j\leq r$
with $X_i^a=X_j$ the characters $\phi_i^a$ and $\phi_j$ coincide.
The equality $X_i^a=X_j$ implies $(X_1^{a_i})^a=X_1^{a_j}$, and hence
$a_i a a_j^{-1}\in \norm{A}{X_1}$. On the other hand by the definition of
$\psi_i$ and $\phi_i$ we have
\[ (\phi_i)^a= (\phi_1)^{a_i a }=
   (\phi_1)^{a_i a a_j^{-1} a_j}= \phi_1^{a_j},\]
since $\phi_1$ is $\norm{A}{X_1}$-invariant. This proves that $\psi$
is $A$-invariant as required.

Part (b) follows from these considerations by applying
Theorem~\ref{thm_IBr_simple}.
\end{proof}

\section{A Gallagher type theorem for blocks}\label{sec_Gal_type_thm}

Let $G$ be a finite group and $N\nor G$. P.~X.~Gallagher proved that if
$\theta \in \Irr(N)$ has an extension $\w\theta \in \Irr(G)$, then the map
$\Irr(G/N) \rightarrow \Irr(G|\theta)$ given by
$\o\eta \mapsto \eta \w\theta$ is a bijection, where $\eta\in\Irr(G)$
is the lift of $\o\eta$. (See Corollary (6.17) of \cite{Isa}.) Now we need
a similar theorem for blocks, see Theorem~\ref{lem2_1_central_def_extending}.

For $\theta\in G$ we denote by $\bl(\theta)$ the $p$-block of $G$ containing
$\theta$.

\begin{lem}   \label{lem2_1_surj_map}
 Let $N\nor G$,  $b\in\Bl(N)$ and $\theta \in\Irr(b)$. Assume there exists an
 extension $\w\theta\in\Irr(G)$ of $\theta$. Then the map
 \[\upsilon: \Bl(G/N) \to \Bl(G| b) \text{ given by }\bl(\o \eta) \mapsto
    \bl(\w\theta\eta)\]
 is surjective, where $\eta\in \Irr(G)$ is the lift of $\o\eta\in\Irr(G/N)$ .
\end{lem}

\begin{proof}
According to \cite[Lemma~2.2]{NavarroSpaeth1} we have for every $g\in G$ that
\[ \lambda_{\w\theta\eta}(\Cl_G(g)^+)=\lambda_{\w\theta_L}(\Cl_L(g)^+)
  \lambda_{\o\eta}(\Cl_{G/N}(\o g)^+),\]
where $L$ is defined by $L/N:=\Cent_{G/N}(\o g)$ and $\o g=gN$. This implies
that the blocks $\bl(\w\theta\eta)$ and $\bl(\w\theta\eta')$ coincide for
every two characters $\o \eta,\o\eta'\in\Irr(G/N)$ with
$\bl(\o\eta)=\bl(\o\eta')$. Hence $\upsilon$ is well-defined.

On the other hand, every block of $\Bl(G|b)$ has a character in
$\Irr(G| \theta)$ (by \cite[(9.2)]{Navarro}) and such a character can
be written as $\w\theta\eta$ for some $\o\eta\in\Irr(G/N)$, by Gallagher's
theorem. This proves that $\upsilon$ is surjective.
\end{proof}

In general, the map in Lemma  \ref{lem2_1_surj_map} is not a bijection. (For
instance, if $b$ has a defect group $D$ such that $\cent GD \sbs N$, then it
is well-known that there is a unique block of $G$ covering $b$, see
\cite[Lemma~3.1]{NavarroTiep}. On the other hand, $G/N$ might have many
$p$-blocks. Take, for instance, $G=\SL_2(3)$, $N=Q_8$, $p=2$, and
$\theta \in \Irr(N)$ the irreducible character of degree 2.) Our aim in
this section is to find general conditions which guarantee that the map in
Lemma~\ref{lem2_1_surj_map} is a bijection.

The following statement follows also from Theorem~3.3(d) of \cite{Harris},
where a Morita equivalence between the involved blocks is proven. For
completeness we nevertheless give here an alternative character theoretic
proof.

\begin{lem}   \label{lem2_1_simple}
 Let $N\nor G$ and $b\in\Bl(N)$ with trivial defect group.
 Let $\theta \in\Irr(b)$. Assume there exists an extension
 $\w\theta\in\Irr(G)$ of $\theta$. Then the map
 \[\upsilon: \Bl(G/N) \to \Bl(G| b) \text{ given by }\bl(\o \eta) \mapsto
   \bl(\w\theta\eta)\]
 is a bijection.
\end{lem}

\begin{proof}
By Lemma \ref{lem2_1_surj_map}, we know that $\upsilon$ is surjective.
Let $\o \eta,\o \eta'\in\Irr(G/N)$ with $\bl(\eta)\neq \bl(\eta')$. By
\cite[Exercise (3.3)]{Navarro} there exists some  $\o  g \in (G/N)^0$ with
$\lambda_{\o \eta}(\Cl_{G/N}(\o g)^+)\neq\lambda_{\o \eta'}(\Cl_{G/N}(\o g)^+)$.
Assume there exists some $c \in G$ with $cN=\o g$ with
\[ \lambda_{\w\theta_{L}}(\Cl_L(c)^+)^*\neq 0,\]
where $L$ is defined by $L/N:=\Cent_{G/N}( \o g)$. Then
$\bl(\w\theta\eta)\neq \bl(\w\theta\eta')$, since
\[\lambda_{\w\theta\eta}(\Cl_G(c)^+)=
  \lambda_{\w\theta_L}(\Cl_L(c)^+)\lambda_{\o\eta}(\Cl_{G/N}(\o g)^+). \]
Thus, let $g\in G$ with $gN=\o g$. Note that $gN$ is closed under
$L$-conjugation, and let $\sS$ be a representative set of the
$L$-conjugacy classes contained in $gN$, i.e.
\[ \bigcup_{s\in\sS}^. \Cl_L(s)=gN.\]
By \cite[Lemma (8.14)]{Isa} we have
\[ \sum_{c \in gN}\w \theta(c) \w\theta(c^{-1})=|N|. \]
This implies
\[ \sum_{s \in \sS} |\Cl_L(s)| \w \theta(s) \w\theta(s^{-1})=|N|. \]
Dividing by $\theta(1)$ we obtain
\[ \sum_{s \in \sS} \lambda_{\w\theta_L}(\Cl_L(s)^+)^* \w\theta(s^{-1})
  =\Big(\sum_{s \in \sS} \frac{|\Cl_L(s)| \w \theta(s)}{\theta(1)}
   \w\theta(s^{-1})\Big)^*=\left(\frac{|N|}{\theta(1)}\right )^*\neq 0 \]
since $\theta$ is of defect~0. This implies that for some $c\in\sS$ we
have $\lambda_{\w\theta_L}(\Cl_L(c)^+)^*\neq 0$ as required.
\end{proof}

Next, we generalize Lemma~\ref{lem2_1_simple}. Recall the
main result of \cite{Nav04}: for a finite group $X$, $Y\nor X$ and a character
$\nu\in\Irr(Y)$ there is a natural bijection $\dz(X) \rightarrow \rdz(X|\nu)$,
$\chi \mapsto \chi_\mu$, with $\frac{\chi(1)_p}{\nu(1)_p}=\frac{|X|_p}{|Y|_p}$,
where $\dz(X)$ denotes the characters lying in a defect zero block and
$\rdz(X| \nu)$ the set of characters $\chi\in\Irr(X| \nu)$,
see \cite{Nav04} for the definition of this bijection.

\begin{lem}  \label{lem_gab}
 Let $N\nor G$, and suppose that $D \nor G$ is contained in $N$. Let
 $\mu \in \Irr(D)$ be $G$-invariant. Let $\theta \in \dz(N/D)$.
 Then $\theta$ extends to $G$ if $\theta_\mu$ extends to $G$.
\end{lem}

\begin{proof}
Suppose that $\theta_\mu$ extends to $G$. We want to show that $\theta$
extends to $G$. It is enough to show that $\theta$ extends to $Q$, whenever
$Q/N$ is a Sylow $q$-subgroup of $G/N$. If $Q/N \in \Syl_p(G/N)$, then $\theta$
considered as a character of $N/D$ has defect zero, and therefore it
extends to $Q/D$ in this case (see, for instance, Problem~(3.10) of
\cite{Navarro}). So $\theta$ as a character of $N$ extends to $Q$. Now suppose
that $Q/N$ is a $q$-group for some $q \ne p$ and recall the bijection
$\dz(Q/D) \rightarrow \rdz(Q|\nu)$ from \cite{Nav04}.
We know that $\theta_\mu$ extends to some $\eta \in \Irr(Q)$. Now,
$\eta(1)_p=\theta_\mu(1)_p$, and therefore $\eta \in \rdz(Q|\mu)$. Then
$\eta=\gamma_\mu$ for some $\gamma \in \dz(Q/D)$. By the values of the
functions $\gamma_\mu$ and using that $\hat\mu (g) \ne 0$ whenever $g_p \in D$
we check that $\gamma_N=\theta$.
\end{proof}

\begin{thm}   \label{lem2_1_central_def_extending}
 Let $N\nor G$ and $b\in\Bl(N)$ with a defect group $D$ with $D\cent GD=G$.
 Let $\theta \in\Irr(b)$. Assume there exists an extension $\w\theta\in\Irr(G)$
 of $\theta$. Then the map
 \[\upsilon: \Bl(G/N) \to \Bl(G| b) \text{ given by }\bl(\o \eta) \mapsto
   \bl(\w\theta\eta)\]
 is a bijection.
\end{thm}

\begin{proof}
By Lemma~\ref{lem2_1_surj_map} it is sufficient to prove that
$|\Bl(G/N)|=|\Bl(G|b)|$.

By \cite[Thm.~(9.12)]{Navarro}, there exists a unique character
$\theta_1\in\Irr(N)$ of $b$ with $D\sbs \ker{\theta_1}$. By Lemma~\ref{lem_gab}
the character $\theta_1$ extends to $G$. Write  $\o N=N/D$ and $\o G=G/D$.
Then Lemma \ref{lem2_1_simple} applies to the character
$\o \theta_1\in\Irr(N/D)$ associated to $\theta_1$. Then
$|\Bl(\o G/\o N)|=|\Bl(\o G| \o b)|$, where $\o b=\bl(\o \theta_1)$.
Now by \cite[Thm.~(9.10)]{Navarro} there is a canonical
bijection between the blocks of $\o G$ and the blocks of $G$, given by
domination. Using \cite[Thm.~(9.2)]{Navarro} with Brauer characters,
we easily check that under this bijection, a block $\o B$ of $\o G$ covers
$\o b$ if and only if $B$ covers $b$. This proves the statement.
\end{proof}

\section{Dade's Ramification Group}\label{sec_Dade_group}
In order to further generalize Theorem~\ref{lem2_1_central_def_extending}, we
need to go deeper and use a subgroup with remarkable properties introduced
by E.~C.~Dade in 1973, see \cite{Dade_blockextensions}. This subgroup is key
in the remainder of this paper.
We shall use M.~Murai's version of it (see \cite{Murai_Dade}).

Suppose that $N \nor G$ and that $\theta \in \Irr(N)$ is $G$-invariant. If
$x, y \in G$ are such that $[x,y] \in N$, then Dade and Isaacs  defined a
complex number $\langle\langle x,y\rangle\rangle_\theta$, in the following
way:  since $N\langle y\rangle/N$ is cyclic, it follows that  $\theta$ extends
to some $\psi \in \Irr(N\langle y\rangle)$. Now, $\psi^x$ is some other
extension and by Gallagher's theorem, there exists some
$\lambda \in \Irr(N\langle y\rangle)$ such that $\psi^x=\lambda \psi$.
Now $\langle\langle x,y\rangle\rangle_\theta = \lambda(y)$. The properties of
this (well-defined) number are listed in Lemma (2.1) and Theorem (2.3) of
\cite{Isa_sympl} which essentially assert that
$\langle\langle,\rangle\rangle_\theta$ is multiplicative (in both arguments).
Thus, if $H, K$ are subgroups of $G$ with $[H,K] \sbs N$, then we have uniquely
defined a subgroup
$$H^\perp \cap K:=\{ k \in K \, |\, \langle\langle k,h\rangle\rangle_\theta=1\,
   \text{ for all } \, h \in H \} \, .$$

\begin{lem}   \label{perp}
 Suppose that $N \nor G$ and that $\theta \in \Irr(N)$ is $G$-invariant.
 Suppose that $H, K$ are subgroups of $G$ with $[H,K] \sbs N$, and $N \sbs K$.
 If $\rho \in \Irr(K)$ extends $\theta$, then $\rho_{H^\perp \cap K}$ is
 $H$-invariant.
\end{lem}

\begin{proof}
Notice that $N \sbs H^\perp \cap K=:M$.  Let $\nu=\rho_{M}$.
We claim that $\nu$ is $H$-invariant. (Since $[H,K] \sbs N$, notice that $H$
normalizes every subgroup between $N$ and $K$.) Now, let $h \in H$ and
$m \in M$. We want to show that $\nu^h(m)=\nu(m)$. Let $J=N\langle m\rangle$,
and write $(\nu_J)^h=\lambda \nu_J$, where $\lambda \in \Irr(J/N)$. Then
$\lambda(m)=\langle\langle h, m\rangle\rangle_\theta=1$, and $\nu^h(m)=\nu(m)$.
\end{proof}

For any block $b\in\Bl(N)$ with defect group $D$, we call a character
$\eta\in\Irr(D\cent N D)$ canonical character of $b$ if $D\sbs \ker\eta$
and $\bl(\eta)^N=b$, see \cite[p.~204]{Navarro}.

\begin{thm}   \label{Dade}
 Let $N \nor G$, and let $b \in \Bl(N)$ be $G$-invariant. Then there exists
 a subgroup $N \sbs G[b] \nor G$, uniquely determined by $G$ and $b$,
 satisfying the following properties.
 \begin{enumerate}[\rm(a)]
  \item Suppose that $D$ is any defect group of $b$, and let
   $\eta \in \Irr(D\cent ND)$ be a canonical character of $b$. Let $K$ be
   the stabilizer of $\eta$ in $D\cent GD$, and let $H$ be the stabilizer of
   $\eta$ in $\norm ND$. Then $G[b]=N(H^\perp \cap K)$.
  \item If $B \in \Bl(G)$ covers $B' \in \Bl(G[b])$ and $B'$ covers $b$, then
   $B$ is the only block of $G$ that covers $B'$.
 \end{enumerate}
\end{thm}

\begin{proof}
We have that $\norm ND$ and $D \cent GD$ are normal subgroups of $\norm GD$
whose intersection is $D\cent ND$. In particular, $[H, K] \sbs D\cent ND$ for
every pair of subgroups $H \sbs \norm ND$ and $K \sbs D\cent GD$.
Now, part~(a) follows from \cite[Thm.~3.13]{Murai_Dade}, while part~(b)
follows from \cite[Thm.~3.5]{Murai_Dade}.
\end{proof}

The subgroup $G[b]$ has many further properties, but we restrict ourselves to
those that shall be used in this paper. For our later considerations we
mention the following.

\begin{cor}   \label{Dade_cor}
 Let $N \nor G$, and let $b \in \Bl(N)$ be $G$-invariant.
 For  $Z\nor G$ with $N\cap Z=1$  let $\o b\in \Bl(NZ/Z)$ be the induced block.
 Then $G[b]/Z= \o G[\o b]$ for $\o G:= G/Z$.
\end{cor}

\begin{proof}
From the assumptions we see $Z\sbs \cent{G}N$ and so $Z\sbs G[b]$ by
Theorem~\ref{Dade}(a).
Let $D$ be a defect group of $b$, $\eta$ a canonical character of $b$,
$K$ the stabilizer of $\eta$ in $D\cent GD$ and $H$ the stabilizer of $\eta$
in $\norm N D$. By Theorem~\ref{Dade}(a) we have $G[b]:=N(H^\perp \cap K)$.

Now $DZ/Z$ is a defect group of $b$, $\o\eta$ is the canonical character of
$b$ induced by $\eta$, $\o K:=KZ/Z$ the stabilizer of $\o \eta$ in
$\o D\cent {\o G} {\o D}$ and $\o H:=HZ/Z$ the stabilizer of $\o \eta$ in
$\norm {NZ/Z} {\o D}$. For $k \in K $ and $h\in H$ we have
\[\langle \langle k ,h \rangle \rangle _{ \eta}
   = \langle \langle k Z ,h Z \rangle \rangle _{\o\eta}. \]
This leads to $\o H^\perp\cap \o K= (H^\perp\cap K) Z/Z$. Together with
Theorem~\ref{Dade}(a) this implies the statement.
\end{proof}

The properties of $G[b]$ allow us to generalize
Theorem~\ref{lem2_1_central_def_extending} to the following situation.

\begin{thm}   \label{thm_block_bijection}
 Let $N\nor G$ and $b\in\Bl(N)$.
 Let $\theta \in\Irr(b)$. Assume there exists an extension $\w\theta\in\Irr(G)$
 of $\theta$. If $G[b]=G$,  then the map
 \[\upsilon: \Bl(G/N) \to \Bl(G| b) \text{ given by }\bl(\o \eta) \mapsto
    \bl(\w\theta\eta)\]   is a bijection.
\end{thm}

\begin{proof}
By Lemma~\ref{lem2_1_surj_map} it is enough to show that
$|\Bl(G|b)| =|\Bl(G/N)|$. By Theorem~\ref{Dade}(a), we have that
$G=N\cent GD$, where $D$ is any defect group of $b$. Let $b'$ be any block
of $D\cent ND$ with defect group $D$ inducing $b$. By \cite[Thm.~C(a.2)]
{KoshitaniSpaeth1},
there exists some $\theta' \in \Irr(b')$ that extends to $K=D\cent GD$.
By Lemma~\ref{lem_gab} the unique (canonical) character $\eta \in \Irr(b')$
that has $D$ in its kernel extends to some $\w\eta \in \Irr(K)$.
Let $T$ be the stabilizer of $\eta$ in
$\norm GD$, hence $D\cent GD = K \sbs T \sbs \norm GD$.
Let $H=T\cap N$. Accordingly $[H,K]\subseteq D\cent N D$ and $T=KH$.
By Theorem~\ref{lem2_1_central_def_extending},
we have that
$$|\Bl(G/N)|=|\Bl(D\cent GD/D\cent ND)|=|\Bl( D\cent GD |b')|.$$
Now by Lemma~\ref{perp}, we have that $\w\eta$ is $H$-invariant. Hence, we
conclude that $T$ is the stabilizer of every block of $H$ covering $b'$
(using Gallagher's theorem). Now, let $\tilde b=(b')^{\norm ND}$.
By the Harris-Kn\"orr correspondence \cite[Thm.~(9.28)]{Navarro} we
have $|\Bl(G|b)|=|\Bl(\norm G D|\tilde b)|$.
If $\{e_1, \ldots, e_s\}$ are the blocks of $D\cent GD$ covering $b'$,
then $\{e_1^{\norm GD}, \ldots, e_s^{\norm GD}\}$ are all the blocks of
$\norm GD$ covering $\w b$. Now, if $(e_i)^{\norm GD}=(e_j)^{\norm GD}$,
then it follows that $(e_i)^x=e_j$ for some $x \in \norm GD$.
Since $e_i$ and $e_j$ only cover the block  $b'$, it follows that $(b')^x=b'$,
and $x \in T$.  However $e_i$ is $T$-invariant, and therefore
$e_i=e_j$.
\end{proof}

\section{Character triple isomorphisms under coprime actions and
  blocks}   \label{sec_char_trip}
In considerations using character triples the existence of an isomorphic
character triple whose character is linear and faithful plays an important
role. We give here an $A$-version of this statement that will be used later.
Furthermore we analyze how blocks behave under the bijections of characters
If $A$ acts on $G$, then $\Irr_A(G)$ denotes the set of the irreducible
complex characters of $G$ which are invariant under $A$. Let $\IBr_A(G)$ be
defined analogously.

\begin{prop}   \label{prop4_1}
 Let $A$ act on $G$ with $(|G|,|A|)=1$ and $N\nor G$ be $A$-stable.
 \begin{enumerate}[\rm (a)]
  \item  Let $\theta\in\Irr_A(N)$. Then there exists a character triple
   $(G^*,N^*,\theta^*)$, an action of $A$ on $G^*$ stabilizing $N^*$ and
   $\theta^*$ and an isomorphism
   $$(\iota,\sigma): (G,N,\theta)\longrightarrow (G^*,N^*,\theta^*),$$
   such that $N^*\sbs \bZ(G ^*)$ and both $\iota$ and
   $\sigma_{G}$ are $A$-equivariant.
  \item Let $\theta\in\IBr_A(N)$. Then there exists a modular character
   triple $(G^*,N^*,\theta^*)$, an action of $A$ on $G^*$ stabilizing $N^*$
   and $\theta^*$ and an isomorphism
   $$(\iota,\sigma): (G,N,\theta)\longrightarrow (G^*,N^*,\theta^*),$$
   such that $N^*\sbs \bZ(G^*)$ is a $p'$-group and both $\iota$
   and $\sigma_{G}$ are $A$-equivariant.
 \end{enumerate}
\end{prop}

\begin{proof}
Since the proof of both statements is based on the same ideas,
we give here only the proof of (a).

First note that $\theta$ extends to $NA$ because of  \cite[Cor.~(6.28)]{Isa}.
By the assumptions it is clear that $(GA,N,\theta)$ is a character triple.
Let $\cP$ be a projective representation of $GA$ with the following
properties:
\begin{enumerate}[\rm (i)]
\item $\cP_{NA}$ affords an extension of $\theta$ to $NA$,
\item the values of the associated factor set
  $\alpha: GA\times GA\rightarrow \CC$ are roots of unity and
\item $\alpha$ is constant on $N \times N$-cosets.
\end{enumerate}
Let $X:=GA$ and $E$ be the finite cyclic group generated by the values of
$\alpha$. For the construction of $ (G^*,N^*,\theta^*)$ and the $A$-action
on $G^*$ we follow the proof of Theorem~(8.28) in \cite{Navarro}.
Let $\w X$ be constructed as there using $\cP$:
the group $\w X$ consists of pairs $(x,\epsilon)$
with $x\in X$ and $\epsilon \in E$ and multiplication in $\w X$ is given by
\[(x_1,\epsilon_1)(x_2,\epsilon_2)
  =(x_1x_2,\alpha(x_1,x_2)\epsilon_1\epsilon_2).\]
The projective representation $\cP$ lifts to a representation of $\w X$.
Let $\tau$ be the character afforded by that representation.
The groups $\w N:=\{(n,\epsilon)\mid  n \in N\}$ and
$\w G:=\{(g,\epsilon)\mid  g \in G\}$ are normal in $\w X$. By the properties
of $\cP_{NA}$ the set $\{(a,1)\mid a \in A\}$ forms a group isomorphic to
$A$. Via this identification $A$ acts on $\w G$ and $\w N$.
Let $E_0:=1\times E$ and $N_0:=N\times 1$. Identifying $N$ and $N_0$ we set
$\w\theta=\theta\times 1_E$. This character is $A$-invariant.

Via the epimorphism $\iota_0: \w X\rightarrow X$, $(x,\epsilon)\mapsto x$, the
character triples $(X,N,\theta)$ and $(\w X, N\times E,\w\theta)$ are
isomorphic.

The map $\w \lambda \in\Irr(\w N)$ with $\w \lambda (n,\epsilon)=\epsilon^{-1}$
is a linear character with kernel $N_0$. By the construction of $\w X$ the
character $\w \lambda$ is $\w X$-invariant. We see that
$\tau_{\w N}=\w \lambda^{-1} \w\theta$.

Now one can argue that $(\w X,\w N,\w\lambda)$ and $(\w X,\w N,\w \theta)$ are
isomorphic character triples. Analogously  $(\w X,\w N,\w\lambda)$ and
$(\w X/N_0,\w N/N_0, \lambda)$ are isomorphic character triples, where
$\lambda\in\Irr(\w N/N_0)$ is the character induced by $\w \lambda$.

With $G^*:=\w G/N_0$, $N^*:=\w N/N_0$ and $\theta^*:=\lambda$ we obtain the
required isomorphism
$$(\iota,\sigma): (G,N,\theta)\longrightarrow (G^*,N^*,\theta^*).$$
Let $\iota_2: \w X \rightarrow \w X/N_0$ be the canonical epimorphism.
Because of $\ker{\iota_1}=E$ and $\ker{\iota_2}=N_0$ the isomorphism $\iota$
can be constructed from $\iota_1$ and $\iota_2$.

For $\chi\in\Irr(G|\theta)$ the character $\sigma_G(\chi)$ is obtained in the
following way: the character $\chi$ lifts to some $\tau_{\w G} \mu$ for some
$\mu\in\Irr(\w G)$ with $\ker\mu\geq N$ and $\mu\in\Irr(\w G|\w \lambda)$.
Hence $\mu \circ \iota_2^{-1}$ is a character of $\Irr(\w G/N_0| \lambda)$.
By its construction $\theta^*$ is $A$-invariant and the maps $\iota$ and
$\sigma_G$ are $A$-equivariant, since $\tau_{\w G}$ is $A$-invariant.
\end{proof}

In order to include blocks in the above result, additional assumptions are
required.

\begin{prop}   \label{prop:Aisom_char_trip}
 Let $N\nor G$ and $b\in\Bl(N)$. Suppose that $A$ acts on $G$ with
 $(|G|,|A|)=1$, such that $N$ and $b$ are $A$-stable. Assume $G[b]=G$.
 \begin{enumerate}[\rm (a)]
  \item Let $\theta\in\Irr_A(N)\cap \Irr(b)$. Let  $(G^*,N^*,\theta^*)$
   and $(\iota,\sigma): (G,N,\theta)\longrightarrow(G^*,N^*,\theta^*)$ be
   as in Proposition~\ref{prop4_1}(a). Then two characters
   $\chi_1,\chi_2\in\Irr(G|\theta)$ satisfy $\bl(\chi_1)=\bl(\chi_2)$ if and
   only if $\bl(\sigma_G(\chi_1))=\bl(\sigma_G(\chi_2))$
  \item Let $\theta\in\IBr_A(N)\cap \IBr(b)$. Let  $(G^*,N^*,\theta^*)$ and
   $(\iota,\sigma): (G,N,\theta)\longrightarrow(G^*,N^*,\theta^*)$ be as in
   Proposition~\ref{prop4_1}(b). Then two characters
   $\phi_1,\phi_2\in\IBr(G|\theta)$ satisfy $\bl(\phi_1)=\bl(\phi_2)$ if and
   only if $\bl(\sigma_G(\phi_1))=\bl(\sigma_G(\phi_2))$.
 \end{enumerate}
\end{prop}

\begin{proof}
We continue using the notation introduced in the proof of
Proposition~\ref{prop4_1}.
Let $N_0:=N\times 1\sbs \w X$ and $\theta_0:=\theta\times 1_E\in\Irr(\w N)$.
Let $b_0:=\bl(\theta_0)$. From Corollary~\ref{Dade_cor} we see that
$G[b]= \w G[ b_0]/E_0$, where $b_0:=\bl( \theta_0)$.

For the proof of (a) let $\mu_1$ and $\mu_2 \in\Irr(\w G)$ with
$N_0\subseteq \ker {\mu_1}$ and $N_0\subseteq \ker{\mu_2}$ such that
$\tau_{\w G}\mu_1$ is a lift of $\chi_1$ and $\tau_{\w G} \mu_2$ is a lift
of $\chi_2$. Note that $E_0\subseteq \bZ (\w G)$. According to \cite[(9.9) and
(9.10)]{Navarro} $\bl(\chi_1)=\bl(\chi_2)$ if and only if $\tau_{\w G} \mu_1$
and $\tau_{\w G} \mu_2$ belong to the same block. According to
Theorem~\ref{thm_block_bijection} we see that the characters of $\w G/N_0$
induced by $\mu_1$ and $\mu_2$ are in the same block. This proves the statement.
\end{proof}

As a corollary that might be of independent interest we conclude the following.

\begin{cor}   \label{cor:isom_char_trip}
 Let $N\nor G$ and $b\in\Bl(N)$. Assume there exists some $G$-invariant
 $\theta\in\Irr_A(N)\cap \Irr(b)$ or $\theta\in\IBr_A(N)\cap \IBr(b)$.
 Let $(G^*,N^*,\theta^*)$ and
 $(\iota,\sigma):(G,N,\theta)\longrightarrow(G^*,N^*,\theta^*)$ be as in
 Proposition~\ref{prop4_1}. Let $H:=G[b]$ and $H^*$ the group with
 $\iota(H/N)=H^*/N^*$.
 \begin{enumerate}[{\rm (a)}]
  \item Then $\Bl(G|b) $ is in bijection with $\Bl(H^*| b^*)$, where
   $b^*:=\bl(\theta^*)$.
  \item Two characters $\chi_1,\chi_2\in\Irr(G|\theta)$ satisfy
   $\bl(\chi_1)=\bl(\chi_2)$ if and only if $\bl(\sigma_G(\chi_1))$ and
   $\bl(\sigma_G(\chi_2))$ cover the same block of $H^*$.
  \item Two characters  $\phi_1,\phi_2\in\IBr(G|\theta)$ satisfy
   $\bl(\phi_1)=\bl(\phi_2)$ if and only if $\bl(\sigma_G(\phi_1))$ and
   $\bl(\sigma_G(\phi_2))$ cover the same block of $H^*$.
 \end{enumerate}
\end{cor}

\begin{proof}
Part~(a) follows directly from Proposition~\ref{prop:Aisom_char_trip}.
Parts~(b) and~(c) are applications of Proposition~\ref{prop:Aisom_char_trip}
together with Theorem~\ref{Dade}(b).
\end{proof}

\section{Reduction}

In this section we show how Theorem~\ref{main} is implied by the analogous
statement for the central product of quasi-simple groups given in
Corollary~\ref{cor_dir_prod}. In fact, we will work with the following
slightly more general statement.

\begin{thm}  \label{thmA}
 Let the group $A$ act on the group $G$ with $(|A|,|G|)=1$.
 \begin{enumerate}[\rm(a)]
  \item  Let  $Z$ be an $A$-invariant central $p$-subgroup of $G$, let
   $\nu\in\Irr_A(Z)$ and let $B\in\Bl_A(G)$. Then there exists an
   $A$-invariant character $\chi\in\Irr(B|  \nu)$.
  \item Let $B\in\Bl_A(G)$. Then there exists an $A$-invariant character
   $\phi\in\IBr(B)$.
 \end{enumerate}
\end{thm}

The following well-known result will be used for part (a).

\begin{thm}   \label{thm_induced_A_fixed_const}
 Let the group $A$ act on the group $G$ with $(|A|,|G|)=1$ and let $N\nor G$
 be $A$-stable. Let $\theta\in\Irr_A(N)$. Then there exists an $A$-invariant
 character in $\Irr(G| \theta)$.
\end{thm}

\begin{proof}
This is Theorem (13.28) and Corollary (13.30) of \cite{Isa}.
\end{proof}

For the proof of Theorem~\ref{thmA}(b) we need the following analogue for
Brauer
characters.

\begin{thm}   \label{thm_induced_A_fixed_const_IBr}
 Let the group $A$ act on the group $G$ with $(|A|,|G|)=1$ and let $N\nor G$
 be $A$-stable. Let $\phi\in\IBr_A(N)$. Then there exists an $A$-invariant
 character in $\IBr(G| \phi)$.
\end{thm}

\begin{proof}
We prove this statement by induction on $|G:N|$ and then on $|G:\bZ(G)|$.
We can assume that $\phi$ is $G$-invariant, since otherwise
some $\phi'\in\IBr_A(G_\phi|\phi)$ exists by induction and hence
$\phi'^G\in\IBr_A(G|\phi)$.

Assume there exists an $A$-stable subgroup $K\nor G$ with $N\lneq K \lneq G$.
Then by induction there exists some $A$-invariant character
$\phi'\in \IBr(K| \phi)$ and one in $\IBr(G| \phi')$. Accordingly we can
assume that $G/N$ is a chief factor of $GA$ and hence $G/N$ is the direct
product of isomorphic simple groups, such that $A$ acts transitively on the
factors of $G/N$.

Then $(G,N,\phi)$ forms a modular character triple that is isomorphic to
some $(G^*,N^*,\phi^*)$ according to Proposition~\ref{prop4_1}, such that
$N^*\sbs \bZ(G^*)$ and $p\nmid |N^*|$. Since the isomorphism of the character
triples is $A$-equivariant it is sufficient to prove the statement for
$(G^*,N^*,\phi^*)$. For this it suffices to prove that there exists some
$A$-invariant character in $\IBr(G^*| \theta^*)$. If $G/N$ is a $q$-group
for $p\neq q$, the group $G^*$ is a $p'$-group and the statement follows
immediately from Theorem \ref{thm_induced_A_fixed_const}.
If $G/N$ is a $p$-group, the set $\IBr(G^*|\phi^*)$ is a singleton.

Let $\nu\in\Irr(N^*)$ be the character with $\nu^0=\phi^*$. (Note that $\nu$
is unique since $N^*$ is a $p'$-group.) Since $\nu$ is $A$-invariant there
exists some $A$-invariant $\chi\in\Irr(G^*|\nu)$ by
Theorem~\ref{thm_induced_A_fixed_const}. This character hence belongs to an
$A$-invariant block $B\in\Bl(G^*)$. Further $G^*/N^*$ is the direct product
of isomorphic non-abelian simple groups that are permuted  transitively by $A$.
According to Corollary~\ref{cor_dir_prod} there exists an $A$-invariant
Brauer character in $\IBr(B)$.
\end{proof}

We start proving  Theorem~\ref{thmA} in a series of intermediate results,
working by induction first on $|G/\bZ(G)|$ and second on $|G|$. It is clear
that we may assume that $Z\in \Syl_p(\zent G)$.
\medskip

Let $N\nor G$ such that $G/N$ is chief factor of $GA$. Then by Glauberman's
Lemma, \cite[Lemma (13.8)]{Isa}, there exists $b\in\Bl_A(N)$ such that
$B\in\Bl(G| b)$. (Recall that by \cite[Cor.~(9.3)]{Navarro}, $G$ acts
transitively on the set of blocks of $N$ covered by $b$.)
Another application of Glauberman's Lemma shows that $b$ has an
$A$-invariant defect group $D$.

Next, notice that $Z \sbs N$ in part (a) of Theorem \ref{thmA}. Otherwise, we
have that $NZ=G$, and therefore $G/N$ is a $p$-group.
In particular $B$ is the only block covering $b$ (\cite[Cor.~(9.6)]{Navarro}).
Let $\nu\in\Irr_A( Z)$. Since $b$ has an $A$-invariant character
$\chi_1 \in \Irr(N|\nu_{Z\cap N})$ by induction,
the character $\chi_1\cdot \nu$ defined as in \cite[Section 5]{IMN} has the
required properties and we are done.

Analogously one can argue that $\zent G \sbs N$ since there exists a unique
$\mu\in \Irr(\zent G|\nu)$ in a block of $\zent G$ covered by $B$.

\begin{lem}
 We can assume that $G_b=G$.
\end{lem}

\begin{proof}
By the Fong-Reynolds theorem \cite[Thm.~(9.14)]{Navarro}, there exists a unique
$\w B\in\Bl(G_b| b)$ with $\w B^G=B$. Since $b$ is $A$-invariant, so is
$G_b$. Also, $\w B$ is $A$-invariant by uniqueness. Notice that
$Z \sbs \zent G \sbs G_b$. If $G_b< G$, then by induction $G_b$ has an
$A$-invariant character
$\chi_0\in\Irr(\w B)\cap \Irr_A(G_b| \nu)$ and $\phi_0\in\IBr(\w B)$. Now, the
characters $\chi_0^G$ and $\phi_0^G$ are irreducible, $A$-invariant and belong
to $B$.
\end{proof}

\begin{lem}
 We can assume that $G[b]=G$.
\end{lem}

\begin{proof}
By Theorem \ref{Dade}(b) there exists a unique $B'\in\Bl(G[b]| b)$
with $(B')^G=B$.  Since $G[b]$ is uniquely determined by $b$, we have that
 $G[b]$ is $A$-stable and by uniqueness that $B'$ is $A$-invariant. Note that
$\bZ(G)\sbs G[b]$ by Theorem \ref{Dade}(a).
If $G[b]\neq G$ then we can conclude by induction that there exist some
$A$-fixed $\chi_0\in\Irr(G[b]| \nu)\cap \Irr(B')$ and $\phi_0\in\IBr(B')$.
By Theorem~\ref{thm_induced_A_fixed_const} and
\ref{thm_induced_A_fixed_const_IBr} there exist some
$\chi\in\Irr(G| \chi_0)$ and
$\phi\in\IBr(G| \phi_0)$, respectively that is $A$-fixed. Those characters
belong to $\Bl(G| B')=\Lset{B}$.
\end{proof}

\begin{lem}   \label{prop3_6}
 We can assume that  $N=\bZ(G)$.
\end{lem}

\begin{proof}
By induction on $|G:\bZ(G)|$ we see that $b$ contains
an $A$-invariant character $\theta\in\Irr(b|\nu)$,
respectively $\theta\in\IBr(b)$. By the above we can
assume $G[b]=G$.

Let $(G^*,N^*,\theta^*)$ be the character triple associated to $(G,N,\theta)$
and $\sigma_G$ the $A$-equivariant bijection
$\Irr(G|\theta)\rightarrow \Irr(G^*|\theta^*)$ from
Proposition~\ref{prop4_1}(a), respectively the $A$-equivariant bijection
$\IBr(G|\theta)\rightarrow \IBr(G^*|\theta^*)$  from
Proposition~\ref{prop4_1}(b). According to
Proposition~\ref{prop:Aisom_char_trip}(b) there is some block $C\in\Bl(G^*)$
such that $\sigma_G(\Irr(B| \theta))= \Irr(C| \theta^*)$ or
$\sigma_G(\IBr(B|\theta))= \IBr(C|\theta^*)$, respectively.

If $N\neq \bZ(G)$ then $\Irr(C| \theta^*)$ and $\IBr(C|\theta^*)$ both
contain $A$-invariant characters because $|G^*:\bZ(G^*)|= |G:N|< |G:\bZ(G)|$.
Since $\sigma_G$ is $A$-equivariant this proves the statement.
\end{proof}

\begin{proof}[Proof of Theorem~\ref{thmA}]
Now it remains to consider the case where $\bZ(G)=N$. Since $N$ was chosen
such that $G/N$ is a chief factor of $GA$, the quotient $G/N$ is the direct
product of isomorphic simple groups that are transitively permuted by $A$.

If $G/N$ is non-abelian, Corollary~\ref{cor_dir_prod} applies and proves the
statement. Otherwise $G/N$ is an elementary abelian $p$-group or a $p'$-group.
In the first case $\Bl(G|b)$ is a singleton and the statement follows from
Theorems~\ref{thm_induced_A_fixed_const}
and~\ref{thm_induced_A_fixed_const_IBr}. In the latter case $B$ has a central
defect group and the sets $\Irr(B|\nu)$ and $\IBr(B)$ are singletons.
\end{proof}


\end{document}